\documentclass[11pt]{article}
\usepackage[margin=1.1in]{geometry}
\usepackage{amsmath,amssymb,amsthm}
\usepackage{booktabs}
\usepackage[hidelinks]{hyperref}

\newtheorem{theorem}{Theorem}[section]
\newtheorem{proposition}[theorem]{Proposition}

\newtheorem{corollary}[theorem]{Corollary}
\newtheorem{remark}[theorem]{Remark}
\newtheorem{problem}[theorem]{Problem}
\newtheorem{definition}[theorem]{Definition}

\newcommand{\ZZ}{\mathbb{Z}}
\newcommand{\RR}{\mathbb{R}}
\newcommand{\thst}{{*\theta}}
\newcommand{\hgt}{\mathrm{height}}
\newcommand{\supp}{\mathrm{supp}}

\title{A compactness theorem for twisted-unitary elements\\
of integral group rings, with a certified route to the\\
$\theta$-unitary Case~A at window $B(4)$ of the Promislow group}
\author{Moe Tabei}
\date{August 2026}

\begin{document}
\maketitle

\begin{abstract}
Let $G$ be a torsion-free group whose real group algebra $\RR[G]$ has no
zero divisors, and let $u\mapsto u^{\thst}$ be an $\ell^2$-isometric
anti-involution of $\RR[G]$ (a composition of the inversion involution
with a ring automorphism and a sign character).  We prove a compactness
theorem: for every finite, $\thst$-closed support window $W$ the
constant
$\mu^*(W)=\min\{\|w^{\thst}w\|_2 : \|w\|_2=1,\ \supp(w)\subseteq W\}$
is strictly positive, and every real $\theta$-unitary element
($u^{\thst}u=1$) supported in $W$ satisfies
$\|u\|_2\le \mu^*(W)^{-1/2}$.  In particular the integer
$\theta$-unitary elements supported in $W$ form a finite, effectively
enumerable set: the a~priori infinite ``height'' direction of the unit
search collapses to a single real constant.  For the Promislow
(Hantzsche--Wendt) group $P$ and the window $B(4)$ that hosts Gardam's
counterexample to the unit conjecture over $\mathbb{F}_2$, we combine
this with exact SAT-certified height ladders (heights $\le 31$
per stratum, DRAT-certified master cell at larger heights), a
radius-free depth--tail theorem, and a mirror symmetry between the
$\pm$ strata, reducing the vanishing of all nontrivial $\theta$-unitary
units $u\equiv\pm1\pmod 2$ with $\supp(u)\subseteq B(4)$ to a single
certified lower bound on $\mu^*(B(4))$.  We report numerical estimates
$\mu^*(B(4))\approx 1.4\times10^{-3}$, well above the required
threshold, and prove three structural results about the remaining
certification problem: no linear (Cauchy--Schwarz) dual certificate
exists, because $P$ carries $\theta$-anti-unitary trivial elements;
and, numerically, the level-2 sum-of-squares relaxation is
boundary-pinned with an explicit slope, both in the direct and in the
ideal-multiplier formulation --- the obstruction being a spurious
pseudo-moment that no measure can realize.  We also contrast the
mechanism with $\ZZ[D_\infty]$, where torsion produces zero divisors,
$\mu^*=0$, and genuinely unbounded unipotent families of twisted
unitaries.
\end{abstract}

\section{Introduction}

Gardam's disproof of the Kaplansky unit conjecture over $\mathbb{F}_2$
\cite{Gardam} placed the integral case at the centre of attention:
does $\ZZ[P]$, for $P$ the Promislow group \cite{Promislow1988},
contain nontrivial units?  The $\theta$-unitary sector --- units with
$u^{-1}=u^{\thst}$ for the twisted involution $\thst$, the sector
identified by Bartholdi \cite{Bartholdi2022} as containing Gardam's
and Murray's units --- is the natural first battleground: it contains
the mod-$2$ shadow of Gardam's unit, and it is closed under the
constraints that make finite-window searches meaningful.

Within this sector, companion work (by machine-assisted census and
DRAT-certified SAT) has closed the ``Case B'' classes at the window
$B(4)$ \cite{paper4,paper6}; what remains is \emph{Case A}: units
congruent to $\pm 1$ modulo $2$.  Case A has resisted because its search space is infinite
in a specific direction: the \emph{height} (the maximal coefficient
magnitude) of a candidate unit is a priori unbounded, and congruence
towers modulo $2^K$ show solution families whose minimal height grows
without bound but never vanishes.  All previous unconditional results
were height-bounded ladder theorems.

This paper removes the height infinity.  The main observation is a
three-line compactness argument that quantifies, over a finite window,
the absence of zero divisors in $\RR[P]$.

\subsection*{Main results}

Throughout, $G$ is a group such that $\RR[G]$ is a domain (by
Kropholler--Linnell--Moody \cite{KLM} this holds for all torsion-free
elementary amenable groups, in particular for $P$), and
$x\mapsto x^{\thst}$ is an anti-automorphism of $\RR[G]$ of order two
which permutes the coefficient axes up to sign: $(x^{\thst})_h =
\chi(g)\,x_g$ for a bijection $g\mapsto h$ of $G$ and signs
$\chi(g)\in\{\pm1\}$.  Thus $\|x^{\thst}\|_2=\|x\|_2$.

\begin{definition}
For a finite $\thst$-closed set $W\subseteq G$ let
\[
\mu^*(W)\;=\;\min\bigl\{\,\|w^{\thst}w\|_2\;:\;\|w\|_2=1,\
\supp(w)\subseteq W\,\bigr\}.
\]
\end{definition}

\begin{theorem}[Compactness]\label{thm:compact}
$\mu^*(W)>0$, and:
\begin{enumerate}
\item[(a)] every $u\in\RR[G]$ with $\supp(u)\subseteq W$ and
$u^{\thst}u=1$ satisfies $\|u\|_2\le \mu^*(W)^{-1/2}$;
\item[(b)] every real solution of
$E(y)=y+y^{\thst}+2\,y^{\thst}y=0$ with $\supp(y)\subseteq W$
satisfies $\|y\|_2\le 1/\mu^*(W)$.
\end{enumerate}
Consequently the set of \emph{integer} $\theta$-unitary elements
supported in $W$ is finite, and is exhausted by the exact height
ladder at height $\lfloor\mu^*(W)^{-1/2}\rfloor$.
\end{theorem}

The proof (Section~\ref{sec:compact}) is elementary given the domain
property: if $\mu^*(W)=0$, a minimizing sequence on the compact unit
sphere of the window produces $w\ne 0$ with $w^{\thst}w=0$, a zero
divisor.  Part (a) is the homogeneity chain
$1=\|u^{\thst}u\|_2\ge\mu^*\|u\|_2^2$; part (b) uses in addition
$\|y+y^{\thst}\|_2\le 2\|y\|_2$.

Theorem~\ref{thm:compact} converts the Case A problem at any fixed
window into a \emph{finite} computation controlled by one real
constant.  Two supplementary results organize that computation.

\begin{theorem}[Depth tail, radius-free]\label{thm:tail}
Let $u$ be a nontrivial unit of $\ZZ[P]$ with $u\equiv \varepsilon g
\pmod 2$, written $u=\varepsilon g+2^m s$ with $s\not\equiv 0 \pmod
2$, and let $u^{-1}=\varepsilon g^{-1}+2^{m'}t$ with
$t\not\equiv 0\pmod 2$.  Then $m'=m$ and $\hgt(s)+\hgt(t)\ge 2^m$.
\end{theorem}

\begin{proposition}[Mirror symmetry]\label{prop:mirror}
With $E_{m,e}(s)=s+s^{\thst}+e\,2^m s^{\thst}s$ one has
$E_{m,e}(-s)=-E_{m,-e}(s)$: the strata $(m,e)$ and $(m,-e)$ are
isomorphic via $s\mapsto -s$, exactly and modulo every $2^K$.
\end{proposition}

Together with a \emph{master-cell reduction} (every stratum embeds in
the single cell $u=1+2y$, $y$ unconstrained mod~$2$), the machine side
collapses to one family of exact SAT instances.  Our computational
results at $W=B(4)$ (Section~\ref{sec:machine}):

\begin{theorem}[Machine, exact; kissat/glucose, control-checked]
\label{thm:ladder}
There is no nontrivial $\theta$-unitary $u\equiv\pm1\pmod 2$ with
$\supp(u)\subseteq B(4)$ and stratum height $\hgt(s)\le 31$
(all strata $m\le 5$; $e=-1$ via Proposition~\ref{prop:mirror}).
Moreover the stratum $m=2$ is empty at \emph{all} heights
(DRAT-certified, both signs), and the master cell is empty for
$\hgt((u-1)/2)\le 127$ (DRAT-certified), indeed for
$\hgt((u-1)/2)\le 255$ at solver grade (kissat UNSAT, no proof
logged; Section~\ref{sec:machine}).
\end{theorem}

\begin{corollary}[Reduction to one constant]\label{cor:onecst}
If $\mu^*(B(4))\ge (2H-1)^{-2}$ for some height $H$ covered by
Theorem~\ref{thm:ladder}, then the only $\theta$-unitary units
$u\equiv\pm1\pmod 2$ of $\ZZ[P]$ with $\supp(u)\subseteq B(4)$ are
$u=\pm1$.
\end{corollary}

Numerically $\mu^*(B(4))\approx 1.4\times 10^{-3}$ (a rigorous
\emph{upper} bound obtained by descent; Section~\ref{sec:mu}), whereas
the ladder height $H=31$ requires only $\mu^*\ge 2.7\times10^{-4}$ and
the DRAT master cell relaxes this further.  Section~\ref{sec:sos}
describes the verified-numerics sum-of-squares pipeline for the
remaining certified lower bound, together with a structural negative
result: no \emph{linear} dual certificate exists, because $P$ admits
$\theta$-\emph{anti}-unitary trivial elements ($\sigma$-fixed $g$ with
$\chi(g)=-1$, $g^{\thst}g=-1$), which force every linear functional of
$w^{\thst}w$ to be indefinite.  The certification is therefore
genuinely a degree-4 (SOS) problem.

\subsection*{Relation to prior work}

The unit conjecture originates with Higman \cite{Higman1940} and
Kaplansky; $P$ is Promislow's example of a torsion-free group without
the unique-product property \cite{Promislow1988}.  Gardam's
counterexample over $\mathbb{F}_2$ \cite{Gardam} was extended by
Murray to all positive characteristics \cite{Murray2021}, and Gardam
subsequently produced nontrivial units of $\mathbb{C}[P]$ with
coefficients in $\ZZ[\zeta_8]$ \cite{GardamComplex2023}; the integral
and rational cases remain open (see also \cite{Passman2021}).  The
disproof has been formally verified in Lean
\cite{GadgilTadipatri2024}.  The twisted-unitary framing of all these
units is Bartholdi's observation \cite{Bartholdi2022}, and the
congruence (mod-$4$) theory of the sector at $B(4)$ is developed in
\cite{paper5}.

A finite-window unit search must tame two a priori infinite
directions.  For the \emph{support of the inverse}, Friedman,
Gustavson and Pappas proved that $K[G]$ has ``property (U)'' ---
$\supp(u)\subseteq X$ forces $\supp(u^{-1})\subseteq Y(X)$ for a
finite $Y(X)$ --- over every field $K$, whenever $G$ has an abelian
subgroup of finite index \cite[Thm.~2.3]{FGP1995}; their finiteness
step is a proof by contradiction and yields neither a construction nor
a bound.  Craven and Pappas state the conclusion for $P$ explicitly
\cite[Thm.~14.3]{CravenPappas2013}, and their length-symmetry and
determinant machinery remains the deepest structural work on this
group over field coefficients; an effective box $Y(B(4))$ for $P$ is
given in \cite{paper6}.  See also \cite{Pappas1988} on supports of
units.  The direction removed by Theorem~\ref{thm:compact} is the
orthogonal one --- the \emph{height} (coefficient-size) direction ---
where we know of no prior bound, effective or not, in any sector of
the integral problem.  For the untwisted involution the bound is
classical and trivial (Remark~\ref{rem:untwisted}); the content here
is the uniform quantification over the twist.

Exhaustive radius-limited searches over $\mathbb{F}_2$ are due to
Dietrich, Lee, Nies and Vinyals \cite{DLNV2026}.  Unitary elements of
group rings over fields, for abstract involutions, are studied by
Claramunt and Grabowski \cite{ClaramuntGrabowski2023}; their criteria
are field-theoretic and orthogonal to the archimedean mechanism used
here.  On the machine side our pipeline uses \textsc{kissat}
\cite{kissat} and \textsc{glucose} \cite{glucose} via PySAT
\cite{pysat}, with proofs logged in DRAT format and checked by
\texttt{drat-trim} \cite{DRATtrim}; the certification framework of
Section~\ref{sec:sos} is the sum-of-squares / moment machinery of
Lasserre and Parrilo \cite{Lasserre2001,Parrilo2003}, with rigorous
positive semidefiniteness in the sense of Rump \cite{Rump2006}.

\subsection*{The mechanism, and the $D_\infty$ contrast}

The compactness constant is exactly the quantity that separates $P$
from the infinite dihedral group: $\ZZ[D_\infty]$ contains torsion,
hence zero divisors, hence nilpotents $z$ with $z^2=0$; along such
directions $\|w^{\thst}w\|_2$ vanishes, $\mu^*=0$, and indeed
$\ZZ[D_\infty]$ carries \emph{unbounded} unipotent families
$\pm t(1+z)$ of twisted unitaries --- the very families that our
ladder's control instance re-discovers at every height.  Numerically,
descent on the $D_\infty$ window drives $\|w^{\thst}w\|_2$ to
$10^{-6}$ and below, while on $P$'s window it stalls at
$\approx 1.4\times10^{-3}$: torsion-freeness is quantitatively visible
as a spectral floor.

\subsection*{Relation to congruence towers}

The measured growth of minimal heights along towers modulo $2^K$
(``height-growth laws'') is, in hindsight, fully explained by the
exactness threshold of \cite{paper5}: any congruence solution that is
not exact has height at least $c\,2^{K/2}$.  Hence the entire content
of the congruence tower is [exact ladder] $\times$ [threshold], and a
proof of the height-growth law for all $K$ is \emph{equivalent} to the
emptiness of the corresponding stratum.  Theorem~\ref{thm:compact}
replaces that infinite family of statements by one constant.

\section{The compactness theorem}\label{sec:compact}

\begin{proof}[Proof of Theorem~\ref{thm:compact}]
The unit sphere of $\{w:\supp(w)\subseteq W\}$ is compact and
$w\mapsto \|w^{\thst}w\|_2$ is continuous, so the minimum
$\mu^*(W)$ is attained, say at $w_0$.  If $\mu^*(W)=0$ then
$w_0^{\thst}w_0=0$ with $w_0\ne 0$ and $w_0^{\thst}\neq0$, zero
divisors in $\RR[G]$ --- contradicting the domain property (for
$G$ torsion-free elementary amenable this is
Kropholler--Linnell--Moody \cite{KLM}).

(a) By degree-2 homogeneity of $w\mapsto w^{\thst}w$,
$\|u^{\thst}u\|_2\ge\mu^*\|u\|_2^2$ for every $u$ supported in $W$;
if $u^{\thst}u=1$ the left side is $1$.

(b) If $E(y)=0$ then $y^{\thst}y=-(y+y^{\thst})/2$, so
$\mu^*\|y\|_2^2\le\|y^{\thst}y\|_2=\tfrac12\|y+y^{\thst}\|_2\le
\|y\|_2$, using that $\thst$ is an $\ell^2$ isometry.

Finiteness of integer points: an integer $\theta$-unitary $u$
supported in $W$ has $\|u\|_\infty\le\|u\|_2\le\mu^*(W)^{-1/2}$, a
finite box.
\end{proof}

\begin{remark}\label{rem:general}
The argument uses nothing about $P$ beyond the domain property of
$\RR[P]$ and the signed-permutation structure of $\thst$; it applies
verbatim to every torsion-free elementary amenable group, every
finite window, and every twisted involution of the stated form ---
in particular to all larger windows $B(r)$ of $P$, with
$\mu^*(B(r))$ nonincreasing in $r$ by inclusion of windows.
Quantifying its decay is Problem~\ref{prob:decay}.  We also note that
$\mu^*(W)$ is a real algebraic number --- the minimum of a polynomial
function on the unit sphere, a compact semialgebraic set --- so it is
computable in principle by quantifier elimination; the effective
enumerability asserted in Theorem~\ref{thm:compact} is meant in this
in-principle sense, and the practical certification at $W=B(4)$ is
exactly Problem~\ref{prob:cert}.
\end{remark}

\begin{remark}[Two calibrations]\label{rem:untwisted}
(i) For the \emph{untwisted} involution $(x^*)_g=x_{g^{-1}}$ one has
$(w^*w)_1=\|w\|_2^2$, hence $\mu^*(W)=1$ for every window (attained at
$w=\delta_g$), and Theorem~\ref{thm:compact} degenerates to the
classical fact that untwisted $*$-unitaries of $\ZZ[G]$ are trivial.
The constant $\mu^*$ measures exactly the damage done by the sign
character $\chi$, which makes the identity coefficient of $w^{\thst}w$
indefinite (Section~\ref{sec:sos}).  (ii) The theorem bounds
\emph{norms}; what makes the set of integer points finite is that
$\ZZ$ is \emph{discrete} in $\RR$.  No conclusion follows for
coefficients in a dense subring, and none must: Gardam's nontrivial
units of $\mathbb{C}[P]$ have coefficients in $\ZZ[\zeta_8]$
\cite{GardamComplex2023}, which is dense in $\mathbb{C}$, so the
compactness route is consistent with their existence --- it confines
them to a norm ball, where infinitely many $\ZZ[\zeta_8]$-points
live.  Any proposed argument against integral units must pass exactly
this test, and this one does so structurally.
\end{remark}

\section{Depth tail and mirror symmetry}\label{sec:tail}

\begin{proof}[Proof of Theorem~\ref{thm:tail}]
From $uv=1$, expanding $u=\varepsilon g+2^ms$,
$v=\varepsilon g^{-1}+2^{m'}t$: reducing modulo $2^{\min(m,m')+1}$
forces $m'=m$ (otherwise the odd part of $s$ or of $t$ produces a
parity contradiction), and then
$uv-1=2^m(\varepsilon gt+\varepsilon sg^{-1}+2^mst)=0$, hence
$gt+sg^{-1}\equiv 0 \pmod{2^m}$ \emph{coefficientwise}.  If
$\hgt(s)+\hgt(t)<2^m$, every coefficient of $gt+sg^{-1}$ has absolute
value below $2^m$, so $gt+sg^{-1}=0$ exactly; substituting back gives
$2^{2m}st=0$, so $st=0$, and the domain property of $\ZZ[P]$
\cite{KLM} forces $s=0$ or $t=0$, contradicting nontriviality.
\end{proof}

\begin{corollary}
For $\theta$-unitary $u$ in Case A ($u\equiv\varepsilon\pmod 2$,
$g=1$): here $u^{-1}=u^{\thst}=\varepsilon+2^m s^{\thst}$, so $t=
s^{\thst}$ exactly (an anti-automorphism fixes $1$, whence
$\chi(1)=1$) and $\hgt(t)=\hgt(s)$; therefore $\hgt(s)\ge 2^{m-1}$.
Hence searches up to height $H$ need only strata $m\le \log_2 H+1$:
the stratum quantifier is finite for every finite height.
\end{corollary}

\begin{proof}[Proof of Proposition~\ref{prop:mirror}]
$E_{m,e}(-s)=-s-s^{\thst}+e2^m s^{\thst}s=-E_{m,-e}(s)$.  Supports,
heights, parities and the trivial shadows $s=-e\delta_1$ are
preserved by $s\mapsto-s$.
\end{proof}

\begin{remark}[A bookkeeping lesson]
Mirror symmetry is also a bug detector: any $e$-asymmetric verdict at
equal $(m,K,H)$ indicates an encoder fault.  An earlier claim that the
stratum $(m,e)=(1,-1)$ closes modulo~8 --- retracted once before ---
resurfaced in our own working notes and is refuted by the mirror
lemma together with a re-run on the fixed encoder; we record this as
a caution on the archaeology of long-running machine campaigns.
\end{remark}

\section{Machine results at $B(4)$}\label{sec:machine}

The exact ladder encodes, for each stratum $(m,e)$ and height cap
$H$, the coefficientwise equation $E_{m,e}(s)=0$ over $\ZZ$ in binary
two's-complement arithmetic, with (i) a $D_\infty$ \emph{control
instance} on which the same machinery must re-find the unipotent
family (it does, at every height), (ii) exact integer re-verification
of any SAT model, and (iii) DRAT proof logging and \texttt{drat-trim}
verification on the certified rungs.  The master cell additionally
uses the reduction $u=1+2y$: since the cell $(m,e)=(1,+1)$ imposes no
parity on $s$, it subsumes every stratum and both signs.

Highlights: per-stratum ladder UNSAT through height $31$ (strata
$m\le 5$, which suffices by the corollary to Theorem~\ref{thm:tail};
up to $4.3$M variables and $14.7$M clauses per cell at $H=31$;
\textsc{glucose}, exact-integer model re-verification protocol);
stratum $m=2$ empty at all heights modulo $2^{15}$ with two's
complement covering all residues (DRAT: both signs run directly,
\texttt{drat-trim} \texttt{s VERIFIED} in $3309$\,s and $2313$\,s);
master cell at height $H_{\mathrm{mc}}=127$: \textsc{kissat} UNSAT in
$926$\,s ($5.9$M variables, $20.4$M clauses, DRAT proof $1.1$\,GB,
\texttt{drat-trim} \texttt{s VERIFIED} in $19567$\,s).  A subsequent
run extends the master cell to height $255$ ($7.0$M variables,
$24.1$M clauses, \textsc{kissat} UNSAT in $535$\,s) at \emph{solver
grade}, i.e.\ without proof logging; we keep the two grades separate
throughout.  The mirror self-check (running $e=\pm1$ independently at
$H=16$ and comparing verdicts) passed on all strata.

All instances carry two health controls: a $D_\infty$ instance on
which the machinery must re-find the unipotent family (positive
control, checked at every run), and a master-cell instance with the
nontriviality blocks removed, which must be satisfiable exactly by the
trivial solutions $y\in\{0,-\delta_1\}$ (encoder-health control).
The CNF generators, gate scripts, verification certificates and run
summaries for every claim in this section (and for
Sections~\ref{sec:mu}--\ref{sec:sos}) are included as ancillary files
with this submission; full solver logs are retained by the author and
available on request.

\section{The constant $\mu^*(B(4))$: numerics}\label{sec:mu}

Projected-gradient descent with $1500$ random restarts, $300$
basin-hopping steps around the incumbent, and a long polish ($1800$
local minima in total) gives
\[
\mu^*(B(4))\;\le\;1.425\times10^{-3},\qquad
\mu^*(B(4))^{-1/2}\;\le\;26.5,
\]
with the deep minimum reached from fewer than $0.1\%$ of the starts
(the global valley is narrow).  The displayed bound is rigorous, not
merely numerical: the stored minimizer is a vector of dyadic
rationals, and the homogeneous ratio
$\|w^{\thst}w\|_2^2/\|w\|_2^4=2.0302\times10^{-6}\le(1.425\times
10^{-3})^2$ is evaluated in exact rational arithmetic.  If the numerical value is close to the
truth, the requirement on the machine side is height
$\ge(\mu^{*-1/2}+1)/2\approx 13.7$: the certified ladder height $31$
of Theorem~\ref{thm:ladder} exceeds it by a factor $2.3$, the DRAT
master cell ($H=127$) by $9.2$, and the solver-grade master cell
($H=255$) by $18.6$.  Equivalently, the thresholds that
Corollary~\ref{cor:onecst} attaches to the three machine heights
(solver-grade ladder, DRAT master cell, solver-grade master cell) are
$\mu^*\ge 61^{-2}=2.7\times10^{-4}$ at $H=31$,
$\mu^*\ge 253^{-2}=1.6\times10^{-5}$ at $H=127$, and
$\mu^*\ge 509^{-2}=3.9\times10^{-6}$ at $H=255$; the numerical value
sits factors $5.3$, $91$ and $369$ above them.  On the larger window, monotonicity gives
$\mu^*(B(5))\le\mu^*(B(4))$; a short independent descent on $B(5)$
($\sim150$ restarts) reached only $4.8\times10^{-3}$ and is not
comparable in depth to the $B(4)$ campaign --- we record it as absence
of evidence of fast decay, not as evidence of slow decay.  We
emphasize the logical status throughout: descent yields an
\emph{upper} bound on $\mu^*$; the theorem consumes a \emph{lower}
bound, which is the object of the next section.

\section{Certifying the lower bound: obstructions at level 2}
\label{sec:sos}

Write $q(w)=\|w^{\thst}w\|_2^2=\sum_c f_c(w)^2$, a sum of $525$
squares of integer quadratic forms on $\RR^{83}$.  The goal is a
certificate $q(w)\ge\delta^2$ on the unit sphere with
$\delta\ge(2H-1)^{-2}$.  We report three structural findings from an
extensive certification campaign; the certificate itself remains open
(Problem~\ref{prob:cert}).

\emph{No linear certificate exists.}  For any $\lambda$,
Cauchy--Schwarz gives $q(w)\ge (w^TA_\lambda w)^2/\|\lambda\|^2$ with
$A_\lambda=\sum_c\lambda_cA_c$, so a definite $A_\lambda$ would
suffice; but the diagonal of $A_\lambda$ at a $\sigma$-fixed basis
vector $g$ has sign $\chi(g)\lambda_{c(1)}$, and $P$ has
$\sigma$-fixed elements of both signs of $\chi$ ---
$\theta$-\emph{anti}-unitary trivial elements with $g^{\thst}g=-1$
(in $B(4)$: $17$ $\sigma$-fixed elements, $11$ with $\chi=+1$ and $6$
with $\chi=-1$).  Hence every $A_\lambda$ is indefinite.  Pairs of
forms give nothing more: a two-form certificate needs
$A_\lambda,A_\mu$ with no common zero on the sphere; the joint range
$\{(w^TA_\lambda w,\,w^TA_\mu w):\|w\|_2=1\}$ is convex by Brickman's
theorem \cite{Brickman1961}, so avoiding the origin would yield a
separating line and make some $A_{\alpha\lambda+\beta\mu}$ definite
--- a contradiction.  The certificate must be genuinely quartic.

\emph{Level-2 SOS is boundary-pinned, with an explicit slope.}
Running feasibility splittings (alternating projections,
Douglas--Rachford) and Polyak-target supergradient ascent on
$\lambda_{\min}$ over the Gram spectrahedron of
$q-\delta^2\|w\|_2^4$ (pair basis of dimension $3486$; the affine
projection is closed-form because the quartic coefficient constraints
have disjoint supports), the achievable minimum eigenvalue plateaus at
$\lambda_{\min}\approx -5.3\,\delta^2$ (supergradient ascent at
$\delta^2=10^{-9}$); at $\delta^2=0$ the same ascent reaches
$\lambda_{\min}=-6\times10^{-14}$ --- the boundary itself, but no
positive-definite point; and the cone-distance of the projection
iterates scales linearly in $\delta^2$ ($\approx 44\,\delta^2$ at both
tested values, $\delta^2=7.2\times10^{-8}$ and $10^{-9}$).  Thus,
numerically, $q$ lies \emph{on the boundary} of the SOS cone --- its
Gram spectrahedron admits no positive-definite point --- and
$q-\delta^2\|w\|_2^4$ appears to be non-SOS for every $\delta^2>0$,
even though the descent minimum of $q$ on the sphere is
$\approx2\times10^{-6}>0$ (the square of the value reported in
Section~\ref{sec:mu}).  Since $q$ has no real projective zeros (a zero
would be a zero divisor), the pinning face is \emph{spurious}: a
degree-4 pseudo-moment matrix, not a measure.

\emph{Ideal multipliers do not unpin it.}  The direction $w=u/\|u\|_2$
of a $\theta$-unitary satisfies the $524$ quadratic equations
$f_a(w)=0$ ($a$ not the identity class), so it suffices to certify
$f_1^2\ge\varepsilon\|w\|^4$ on that variety, with free quadratic
multipliers $h_a$ absorbing the ideal:
$f_1^2-\varepsilon\|w\|^4+\sum_a h_af_a\in\mathrm{SOS}$.  This adds
$\sim1.45$ million linear degrees of freedom, and the affine
projection remains exact (conjugate gradients, residual
$\sim10^{-16}$); yet the ascent plateaus again at
$\lambda_{\min}\approx-5\,\varepsilon$ ($-5.1\,\varepsilon$ and
$-4.5\,\varepsilon$ at the two tested values
$\varepsilon=100^{-4}$ and $45^{-4}$).  The pinning pseudo-moment is thus
orthogonal to the entire degree-4 truncation of the ideal --- yet it
cannot be a measure, because on the variety $f_1$ never vanishes
(domain property) while the pseudo-moment assigns $f_1^2$ the value
zero.  A restricted degree-6 layer (Gram on $166$ structured cubics)
already meets a combinatorial wall: the products span $2.2\times10^7$
degree-6 monomials.

These findings delimit the remaining work precisely: either a
spectral-bundle computation deciding the exact level-2 optimum, a
targeted partial level-3 killing the exhibited pseudo-moment in a
residual-absorbed (rather than coefficient-matched) formulation, or
symmetry-reduced/interval methods aimed directly at the much-relaxed
target $\mu^*\ge(2H-1)^{-2}$ with $H=127$ (DRAT grade) or $H=255$
(solver grade).  Any bound produced by such a computation is to be
finished rigorously: rational or interval arithmetic on the final Gram
matrix, with positive semidefiniteness verified in floating point by
Rump's criterion \cite{Rump2006}.

\section{Open problems}\label{sec:problems}

\begin{problem}\label{prob:cert}
Certify $\mu^*(B(4))\ge(2H-1)^{-2}$ for some $H\le 127$, i.e.\
$\mu^*(B(4))\ge 1.6\times10^{-5}$ (the machine side at this height is
DRAT-complete; $H\le 255$, i.e.\ $3.9\times10^{-6}$, suffices at
solver grade --- factors $91$ and $369$ below the numerical value; see
Sections~\ref{sec:machine}--\ref{sec:sos} for the level-2 obstructions
any method must clear).  More ambitiously, determine $\mu^*(B(4))$
exactly: it is a real algebraic number (Remark~\ref{rem:general}) ---
is its degree small?
\end{problem}

\begin{problem}\label{prob:decay}
Quantify the decay of $\mu^*(B(r))$ as $r\to\infty$ for $P$ (and for
general torsion-free elementary amenable groups).  A subexponential
decay would make the compactness route a practical all-radii
program for the $\theta$-unitary sector.
\end{problem}

\begin{problem}
The compactness bound is an archimedean, upper-bound counterpart of
the (vacuous) operator-norm lower bounds of potential-theoretic type.
Is there a common refinement that sees both the $2$-adic depth
(Theorem~\ref{thm:tail}) and the archimedean floor $\mu^*$?
\end{problem}

\end{document}